\documentclass[a4,reqno,12pt]{amsart}
\usepackage{verbatim}
\usepackage{amssymb,amsmath,amsthm}
\usepackage{amsfonts}
\usepackage{array}

\newtheorem{theorem}{Theorem}[section]
\newtheorem{lemma}{Lemma}[section]
\newtheorem{corollary}{Corollary}[section]

\begin{document}

\title[On $q$-analogues of two-one formulas]{\small On $q$-analogues of two-one formulas for multiple harmonic sums and multiple zeta star values}

\author{Kh.~Hessami Pilehrood}
\address{Department of Mathematics and Statistics \\ Dalhousie University \\ Halifax, Nova Scotia, B3H 3J5, Canada}
\email{hessamik@gmail.com}

\author{T.~Hessami Pilehrood}
\address{Department of Mathematics and Statistics \\ Dalhousie University \\ Halifax, Nova Scotia, B3H 3J5, Canada}
\email{hessamit@gmail.com}

\begin{abstract}
Recently, the present authors jointly with Tauraso found a family of binomial identities for multiple harmonic sums (MHS)
on strings $(\{2\}^a,c,\{2\}^b)$ that appeared to be useful for proving new congruences for MHS as well as new relations for multiple zeta values.
Very recently, Zhao generalized this set of MHS identities to strings with repetitions of the above patterns and, as an application, proved the two-one formula for multiple zeta star values conjectured by Ohno and Zudilin. In this paper, we extend our approach to $q$-binomial identities and prove
$q$-analogues of two-one formulas for multiple zeta star values.
\end{abstract}

\maketitle

\section{Introduction}

Let $s_1,\ldots, s_m$ be positive integers, $s_1\ge 2.$ The multiple zeta star and zeta values are defined by the convergent series
\begin{align}
\zeta^{\star}({\bf s})&=\zeta^{\star}(s_1, \ldots, s_m)=\sum_{k_1\ge\ldots\ge k_m\ge 1}\frac{1}{k_1^{s_1}\cdots k_m^{s_m}}, \label{zs} \\
\zeta({\bf s})&=\zeta(s_1, \ldots, s_m)=\sum_{k_1>\ldots>k_m\ge 1}\frac{1}{k_1^{s_1}\cdots k_m^{s_m}}. \label{zz}
\end{align}
The origin of these numbers goes back to correspondence of Euler with Goldbach in 1742-43 (see \cite{Ho:slides})  and Euler's paper \cite{Euler}
that appeared in 1776. Euler studied double zeta values and established some important relation formulas for them. For example, he proved that
\begin{equation*}  
2\zeta^{\star}(n,1)=(n+2)\zeta(n+1)-\sum_{i=1}^{n-2}\zeta(n-i)\zeta(i+1), \qquad n\ge 2,
\end{equation*}
which, in particular, implies the simplest but nontrivial relation:
\begin{equation} \label{zeta21}
\zeta^{\star}(2,1)=2\zeta(3) \quad \text{or equivalently,} \quad \zeta(2,1)=\zeta(3).
\end{equation}

The systematic study of multiple zeta values began in the early 1990s with the works of Hoffman \cite{Ho:92} and Zagier \cite{Za:92}
and have continued with increasing attention in recent years (see survey articles \cite{BB}, \cite{Ho:05}, \cite{Wa:00}, \cite{Zu:05}).

There are also a lot of contributions on $q$-analogues of multiple zeta (star) values (see \cite{Br:05}, \cite{OZ:12}, \cite{Zhao:07}).

The purpose of the present paper is to establish  $q$-analogues of the two-one formulas for multiple zeta star values conjectured by Ohno and Zudilin \cite{OZ:08} and   proved very recently by Zhao \cite{Zhao:13}.

Zhao's proof is based on generalizations of binomial identities for finite multiple harmonic sums found in \cite{HP:12} to strings
with repeating collections of twos and ones. In this paper, we extend this approach to $q$-analogues of multiple harmonic sums and, as a limit case, obtain corresponding results for $q$-zeta values.

We begin with some basic notation.
Let $q$ be a real number with $0<q<1.$ The $q$-analogue of a non-negative integer $n$ is defined as
$$
[n]_q=\sum_{k=0}^{n-1}q^k=\frac{1-q^n}{1-q}.
$$
For any real number $a,$ put
$$
(a)_0:=(a;q)_0:=1, \qquad (a)_n:=(a;q)_n:=\prod_{k=0}^{n-1}(1-aq^k), \quad n\ge 1.
$$
Let $n, m$ denote integers. Then the Gaussian $q$-binomial coefficient is defined by
\begin{equation*}
\genfrac{[}{]}{0pt}{}{n}{m}:=\begin{cases}
\frac{(q)_n}{(q)_{m}(q)_{n-m}}, \quad\qquad &\text{if} \,\,\, 0\le m\le n, \\
\qquad 0, \quad\qquad & \text{otherwise.}
\end{cases}
\end{equation*}
For non-negative integers $n, m$ and ${\bf s}=(s_1, \ldots, s_m)\in{\mathbb Z}^m,$ ${\bf t}=(t_1,\ldots, t_m)\in{\mathbb Z}^m,$ define $q$-analogues of multiple harmonic sums
\begin{align*}
H^{\star}_n[{\bf s}]&=H^{\star}_n[s_1, \ldots, s_m]=\sum_{n\ge k_1\ge\ldots \ge k_m\ge 1}\frac{q^{k_1}}{[k_1]_q^{s_1}}
\cdots\frac{q^{k_m}}{[k_m]_q^{s_m}},  \\[2pt]
\mathcal{H}_n[{\bf s}; {\bf t}]&=\mathcal{H}_n[s_1,\ldots,s_m;t_1,\ldots,t_m]=
\sum_{n\ge k_1>\ldots k_m\ge 1}\prod_{j=1}^m\frac{q^{(t_j-1)k_j}(1+q^{k_j})}{[k_j]_q^{s_j}},
\end{align*}
with the convention that $\mathcal{H}_n[{\bf s}; {\bf t}]=0$ if $n<m,$ and
$H^{\star}_n[\emptyset]=\mathcal{H}_n[\emptyset;\emptyset]=1$ for all $n\ge 0$ and $m=0.$

We consider $q$-analogues of  multiple zeta star values (\ref{zs})
and multiple zeta  values (\ref{zz}),
which are defined by
\begin{equation} \label{starz}
\zeta_q^{\star}[{\bf s}]=\zeta_q^{\star}[s_1,\ldots,s_m]:=\sum_{k_1\ge\ldots\ge k_m\ge 1}\frac{q^{k_1}}{[k_1]_q^{s_1}}\cdots \frac{q^{k_m}}{[k_m]_q^{s_m}}
\end{equation}
and
\begin{equation*}
\hat{\mathfrak{z}}_q[{\bf s};{\bf t}]=\hat{\mathfrak{z}}_q[s_1,\ldots,s_m; t_1,\ldots,t_m]=\sum_{k=1}^{\infty}
\frac{q^{k^2+(t_1-1)k}(1+q^k)}{[k]_q^{s_1}}\,\mathcal{H}_{k-1}[s_2,\ldots,s_m; t_2,\ldots,t_m].
\end{equation*}
If $m=0,$ we put $\zeta^{\star}_q[\emptyset]=\hat{\mathfrak{z}}_q[\emptyset;\emptyset]=1.$


\begin{theorem} \label{T1}
Let $m\in {\mathbb N},$ $s_1, \ldots, s_m\in {\mathbb N}_0.$ Then
\begin{equation} \label{Eq16}
\zeta_q^{\star}[\{2\}^{s_1}, 1, \{2\}^{s_2}, 1, \ldots, \{2\}^{s_m},1]=\sum_{{\bf p}=(2s_1+1)\circ(2s_2+1)\circ\ldots\circ(2s_m+1)}
\hat{\mathfrak{z}}_q[{\bf p}; {\bf \widetilde{p}}],
\end{equation}
where $\circ$ is either comma or plus, and the string ${\bf \widetilde{p}}:=(s_1+1)\circ(s_2+1)\circ\ldots\circ(s_m+1)$
is associated with the string ${\bf p}.$ This means that the choice of commas and pluses in ${\bf p}=(2s_1+1)\circ(2s_2+1)\circ\ldots\circ(2s_m+1)$ and ${\bf \widetilde{p}}=(s_1+1)\circ(s_2+1)\circ\ldots\circ(s_m+1)$ is the same.
\end{theorem}

In the next theorem, we consider two-one strings ending with $2.$

Let
\begin{equation*}
\overline{\mathfrak{z}}_q[{\bf s}; {\bf t}]=\overline{\mathfrak{z}}_q[s_1,\ldots,s_m; t_1,\ldots, t_m]=
\sum_{k_1>\ldots>k_m\ge 1}(-1)^{k_m-1}
q^{k_1^2-k_m(k_m-1)/2}
\prod_{j=1}^m\frac{(1+q^{k_j}) q^{(t_j-1)k_j}}{[k_j]_q^{s_j}}.
\end{equation*}

\begin{theorem} \label{T2}
Let $m, s_1, \ldots, s_m\in {\mathbb N}_0,$ $s_{m+1}\in {\mathbb N}.$  Then
\begin{equation} \label{Eq17}
\zeta_q^{\star}[\{2\}^{s_1}, 1,  \ldots, \{2\}^{s_m},1, \{2\}^{s_{m+1}}]=\sum_{{\bf p}=(2s_1+1)\circ\ldots\circ(2s_m+1)\circ(2s_{m+1})}
\overline{\mathfrak{z}}_q[{\bf p}; \widetilde{{\bf p}}],
\end{equation}
where $\circ$ is either comma or plus, and the string ${\bf \widetilde{p}}:=(s_1+1)\circ\ldots\circ(s_m+1)\circ(s_{m+1})$
is associated  with the string ${\bf p}.$ This means that the choice of commas and pluses in ${\bf p}=(2s_1+1)\circ\ldots\circ(2s_m+1)\circ(2s_{m+1})$ and ${\bf \widetilde{p}}=(s_1+1)\circ\ldots\circ(s_m+1)\circ(s_{m+1})$ is the same.
\end{theorem}

Note that the limiting $q\to 1$ case of (\ref{Eq16}) is Ohno-Zudilin's two-one formula from \cite{OZ:08}:
\begin{equation} \label{Eq18}
\zeta^{\star}(\{2\}^{s_1},1,\ldots,\{2\}^{s_m},1)=\sum_{{\bf p}=(2s_1+1)\circ\ldots\circ(2s_m+1)}2^{l({\bf p})}\zeta({\bf p}),
\end{equation}
where $l({\bf p})$ is the length of the string ${\bf p}.$
It is clear that $l({\bf p})=\sigma({\bf p})+1,$
where $\sigma({\bf p})$ is the number of commas in ${\bf p}=(2s_1+1)\circ\ldots\circ(2s_m+1).$

Similarly, taking the limit $q\to 1$ in (\ref{Eq17}) gives Zhao's two-one formula from~\cite[Theorem~1.2]{Zhao:13}.

{\it Remark} Note that the series on the right-hand sides of (\ref{Eq16}) and (\ref{Eq17}) contain quadratic powers of the parameter $q.$
This implies that the series converge rapidly  and therefore, formulas (\ref{Eq16}),  (\ref{Eq17}) can be used for fast calculation of the $q$-zeta star values (\ref{starz}) on strings of twos and ones.

{\it Remark.} It is easy to see that for $m=0$ and $s_{m+1}=0,$ Theorem \ref{T2} is also true. In this case, it corresponds to the trivial telescoping sum:
\begin{equation*}
1=\zeta^{\star}_q[\emptyset]=\overline{\mathfrak{z}}_q[0;0]=\sum_{k=1}^{\infty}(-1)^{k-1}(1+q^k)q^{k(k-1)/2}.
\end{equation*}

In the simplest cases $m=0$ and $m=1,$  Theorem \ref{T2}   and Theorem \ref{T1}   give $q$-analogues of the following well-known formulas
  for ordinary zeta values (see \cite{V}, \cite{Zl}, \cite[p.~292]{OW:06}):
\begin{align}
\zeta^*(\{2\}^s,1)&=2\zeta(2s+1),  \label{Eq19} \\
\zeta^*(\{2\}^s)&=2\sum_{k=1}^{\infty}\frac{(-1)^{k-1}}{k^{2s}}=2(1-2^{1-2s})\zeta(2s).  \nonumber
\end{align}

Note that (\ref{Eq19}) is a particular case of (\ref{Eq18}).
\begin{corollary}
Let $s$ be a nonnegative integer. Then
\begin{align}
\zeta_q^{\star}[\{2\}^{s}]&=\overline{\mathfrak{z}}_q[2s;s]=\sum_{k=1}^{\infty}\frac{1+q^k}{[k]_q^{2s}}\,(-1)^{k-1}q^{k(k-1)/2+sk}, \nonumber\\
\zeta_q^{\star}[\{2\}^{s}, 1]&=\hat{\mathfrak{z}}_q[2s+1;s+1]=\sum_{k=1}^{\infty}\frac{1+q^k}{[k]_q^{2s+1}}\,q^{k^2+sk}. \label{q21}
\end{align}
\end{corollary}

For $s=1,$ from formula (\ref{q21}) we get a new $q$-analogue of Euler's formula (\ref{zeta21}) that becomes
\begin{equation} \label{qz3}
\zeta_q^{\star}[2, 1]=\sum_{k=1}^{\infty}\frac{1+q^k}{[k]_q^{3}}\,q^{k(k+1)}.
\end{equation}
The detailed survey on various generalizations and proofs of formula (\ref{zeta21}) can be found in~\cite{BB:06}.
Note that no extension of Ap\'ery's proof leading to the irrationality of a $q$-analogue of $\zeta(3)$ is known. In this respect,
formula (\ref{qz3}) may be quite helpful: firstly, in view of the fast convergence of the series (\ref{qz3}) and secondly, because of the fact that the irrationality proof of $\zeta(3)$ as a double series, $\zeta(2,1),$ is known (see \cite{So:98}, \cite{So:02}).

\section{$q$-binomial identities}

In this section, we prove some auxiliary $q$-binomial identities, which are $q$-analogues of those proved in \cite[Lemma 2.1]{HP:12}.

\begin{lemma} \label{l1}
For integers $n\ge 1,$ $l\ge 0$ we have
\begin{align}
\sum_{k=l+1}^n(1+q^k)\frac{\genfrac{[}{]}{0pt}{}{n}{k}}{\genfrac{[}{]}{0pt}{}{n+k}{k}}(-1)^{k}
q^{k(k-1)/2}&=\frac{[l]_q-[n]_q}{[n]_q}\frac{\genfrac{[}{]}{0pt}{}{n}{l}}{\genfrac{[}{]}{0pt}{}{n+l}{l}}(-1)^{l}
q^{l(l-1)/2}, \label{i1}\\
\sum_{k=l+1}^n(1+q^k)\,\frac{[k]_q\genfrac{[}{]}{0pt}{}{n}{k}}{\genfrac{[}{]}{0pt}{}{n+k}{k}}
q^{k(k-1)}&=([n]_q-[l]_q)\frac{\genfrac{[}{]}{0pt}{}{n}{l}}{\genfrac{[}{]}{0pt}{}{n+l}{l}}\,
q^{l^2}, \label{i2}\\
\sum_{k=1}^n\frac{1+q^k}{[k]_q}\,\frac{\genfrac{[}{]}{0pt}{}{n}{k}}{\genfrac{[}{]}{0pt}{}{n+k}{k}}\,
q^{k^2}&=\sum_{m=1}^n\frac{q^m}{[m]_q}. \label{i3}
\end{align}
Moreover, if $l\ge 1$ then
\begin{equation} \label{i4}
\sum_{k=l}^n\frac{q^k}{[k]_q^2}\,\frac{\genfrac{[}{]}{0pt}{}{k}{l}}{\genfrac{[}{]}{0pt}{}{k+l}{l}}=
\frac{q^l\genfrac{[}{]}{0pt}{}{n}{l}}{[l]_q^2\genfrac{[}{]}{0pt}{}{n+l}{l}}.
\end{equation}
\end{lemma}
\begin{proof}
It is easy to show that if we put
$$
F(n,k)=(1+q^k)\,\frac{\genfrac{[}{]}{0pt}{}{n}{k}}{\genfrac{[}{]}{0pt}{}{n+k}{k}}\,(-1)^{k-1}
q^{k(k-1)/2}, \qquad G(n,k)=\frac{q^{n+k}-1}{q^k+1}\,F(n,k),
$$
then
\begin{equation}
\label{i11}
(1-q^n)F(n,k)=G(n,k+1)-G(n,k)
\end{equation}
for all positive integers $n, k.$ Summing (\ref{i11}) over $k$ from $l+1$ to $n$ we obtain
\begin{equation*}
\begin{split}
(1-q^n)&\sum_{k=l+1}^nF(n,k)=\sum_{k=l+1}^n(G(n,k+1)-G(n,k))=G(n,n+1)-G(n,l+1) \\
&=-G(n,l+1)
=
\frac{1-q^{n+l+1}}{1+q^{l+1}}\,F(n,l+1)=(1-q^{n-l})\,\frac{\genfrac{[}{]}{0pt}{}{n}{l}}{\genfrac{[}{]}{0pt}{}{n+l}{l}}\,(-1)^{l}
q^{l(l+1)/2},
\end{split}
\end{equation*}
which implies (\ref{i1}).

Similarly, putting
$$
F(n,k)=\frac{\genfrac{[}{]}{0pt}{}{n}{k}}{\genfrac{[}{]}{0pt}{}{n+k}{k}}\,(1-q^{2k})q^{k(k-1)}, \qquad
G(n,k)=\frac{1-q^{n+k}}{1-q^{2k}}\,F(n,k)
$$
we conclude that
\begin{equation} \label{i21}
F(n,k)=G(n,k)-G(n,k+1)
\end{equation}
for all positive integers $n, k.$ Summing both sides of (\ref{i21}) over $k$ from $l+1$ to $n,$ we have
\begin{equation*}
\sum_{k=l+1}^nF(n,k)=
G(n,l+1)-G(n,n+1)=G(n,l+1)=
\frac{\genfrac{[}{]}{0pt}{}{n-1}{l}}{\genfrac{[}{]}{0pt}{}{n+l}{l}}\,(1-q^n)q^{l(l+1)},
\end{equation*}
and the identity (\ref{i2}) follows.

For proving (\ref{i3}), we define for integers $m\ge 0,$ $k\ge 1$ two functions:
$$
F(m,k)=\frac{\genfrac{[}{]}{0pt}{}{m}{k}}{\genfrac{[}{]}{0pt}{}{m+k}{k}}\,\frac{1+q^k}{1-q^k}\,\,q^{k^2}, \qquad
G(m,k)=\frac{q^{m-k+1}}{1+q^k}\,\frac{1-q^k}{1-q^{m-k+1}}\,F(m,k).
$$
Then it is readily seen that
\begin{equation} \label{i31}
F(m,k)-F(m+1,k)=G(m,k+1)-G(m,k).
\end{equation}
Summing both sides of (\ref{i31}) over $m$ from $0$ to $n-1$ we get
\begin{equation} \label{i32}
\sum_{m=0}^{n-1}(G(m,k+1)-G(m,k))=\sum_{m=0}^{n-1}(F(m,k)-F(m+1,k))=F(0,k)-F(n,k)=-F(n,k).
\end{equation}
Summing once again both sides of (\ref{i32}) over $k$ from $1$ to $n,$ we easily obtain
\begin{equation*}
\begin{split}
\sum_{k=1}^nF(n,k)&=\sum_{m=0}^{n-1}\sum_{k=1}^n(G(m,k)-G(m,k+1)) \\
&=\sum_{m=0}^{n-1}(G(m,1)-G(m,n+1)) \\
&=\sum_{m=0}^{n-1}G(m,1)=\sum_{m=0}^{n-1}\frac{q^{m+1}}{1-q^{m+1}},
\end{split}
\end{equation*}
which implies (\ref{i3}).

To prove (\ref{i4}), we put
$$
F(l,k)=\frac{q^{k-l}}{[k]_q^2}\,\frac{\genfrac{[}{]}{0pt}{}{k}{l}}{\genfrac{[}{]}{0pt}{}{k+l}{k}}, \qquad
G(l,k)=\frac{(1-q^{k-l})(1-q^{k+l})}{q^{k-l}}\,F(l,k).
$$
Then it is easy to see that
\begin{equation} \label{i41}
(1-q^l)^2F(l,k)=G(l,k+1)-G(l,k).
\end{equation}
Summing (\ref{i41}) over $k$ from $l$ to $n$ we get
\begin{equation*}
\begin{split}
(1-q^l)^2\sum_{k=l}^nF(l,k)&=\sum_{k=l}^n(G(l,k+1)-G(l,k))=G(l,n+1)-G(l,l) \\
&=G(l,n+1)=(1-q)^2 \frac{\genfrac{[}{]}{0pt}{}{n}{l}}{\genfrac{[}{]}{0pt}{}{n+l}{l}},
\end{split}
\end{equation*}
and the identity (\ref{i4}) follows.
\end{proof}

\section{Identities for multiple harmonic sums}

In this section, we prove $q$-binomial identities for multiple harmonic sums whose indices are strings of twos and ones. We begin with some special
cases and then extend them to arbitrary strings of twos and ones.

\begin{theorem} \label{t1}
Let $a, b$ be integers satisfying $a\ge 0,$ $b\ge 1.$ Then for any positive integer~$n,$
\begin{equation}
H^{\star}_n[\{2\}^a]=\sum_{k=1}^n\frac{1+q^k}{[k]_q^{2a}}\frac{\genfrac{[}{]}{0pt}{}{n}{k}}{\genfrac{[}{]}{0pt}{}{n+k}{k}}
(-1)^{k-1}q^{k(k-1)/2+ak},
\label{t1(2)}
\end{equation}
\begin{equation}
H^{\star}_n[\{2\}^a,1]=\sum_{k=1}^n\frac{1+q^k}{[k]_q^{2a+1}}\,\frac{\genfrac{[}{]}{0pt}{}{n}{k}}{\genfrac{[}{]}{0pt}{}{n+k}{k}}\,\,
q^{k^2+ak},
\label{t2(1)}
\end{equation}
\begin{equation}
\begin{split}
H^{\star}_n[\{2\}^a,1,\{2\}^b]&=-\sum_{k=1}^n
\frac{1+q^k}{[k]_q^{2(a+b)+1}}\frac{\genfrac{[}{]}{0pt}{}{n}{k}}{\genfrac{[}{]}{0pt}{}{n+k}{k}}(-1)^{k}
q^{k(k+1)/2+(a+b)k} \\
&-\sum_{k=1}^n\frac{1+q^k}{[k]_q^{2a+1}}\frac{\genfrac{[}{]}{0pt}{}{n}{k}}{\genfrac{[}{]}{0pt}{}{n+k}{k}}
q^{k^2+ak}\sum_{j=1}^{k-1}\frac{(-1)^j(1+q^j)q^{bj-j(j+1)/2}}{[j]_q^{2b}}.
\label{t2(2)}
\end{split}
\end{equation}
\end{theorem}
\begin{proof}
We show that (\ref{t1(2)}) is true by induction on $a.$ For $a=0$ the equality follows from~(\ref{i1}).
Suppose that the formula is true for $a>0.$ Then by the induction assumption  and identity~(\ref{i4}) we easily conclude that
\begin{equation*}
\begin{split}
H^{\star}_n[\{2\}^{a+1}]&=\sum_{k=1}^n\frac{q^k}{[k]_q^2}\,H^{\star}_k[\{2\}^a]=\sum_{k=1}^n\frac{q^k}{[k]_q^2}\sum_{l=1}^k
\frac{1+q^l}{[l]_q^{2a}}\,\frac{\genfrac{[}{]}{0pt}{}{k}{l}}{\genfrac{[}{]}{0pt}{}{k+l}{l}}\,(-1)^{l-1}
q^{l(l-1)/2+al} \\[3pt]
&=\sum_{l=1}^n\frac{1+q^l}{[l]_q^{2a}}(-1)^{l-1}q^{l(l-1)/2+al}\sum_{k=l}^n\frac{q^k}{[k]_q^2}\,
\frac{\genfrac{[}{]}{0pt}{}{k}{l}}{\genfrac{[}{]}{0pt}{}{k+l}{l}} \\[3pt]
&=\sum_{l=1}^n\frac{1+q^l}{[l]_q^{2a+2}}\,
\frac{\genfrac{[}{]}{0pt}{}{n}{l}}{\genfrac{[}{]}{0pt}{}{n+l}{l}}\,
(-1)^{l-1}
q^{l(l-1)/2+(a+1)l}
\end{split}
\end{equation*}
and the formula is proved.

We prove the second identity also by induction on $a.$ For $a=0$ its validity follows from (\ref{i3}).
Assume the formula holds for $a>0.$ Then by the induction assumption  and formula~(\ref{i4}), we easily obtain
\begin{equation*}
\begin{split}
H^{\star}_n[\{2\}^{a+1},1]&=\sum_{k=1}^n\frac{q^k H^{\star}_k[\{2\}^a,1]}{[k]_q^2}=\sum_{k=1}^n\frac{q^k}{[k]_q^2}\sum_{l=1}^k
\frac{1+q^l}{[l]_q^{2a+1}}\frac{\genfrac{[}{]}{0pt}{}{k}{l}}{\genfrac{[}{]}{0pt}{}{k+l}{l}}\,
q^{l^2+al} \\[3pt]
&=\sum_{l=1}^n\frac{1+q^l}{[l]_q^{2a+1}}\,q^{l^2+al}\sum_{k=l}^n\frac{q^k}{[k]_q^2}
\frac{\genfrac{[}{]}{0pt}{}{k}{l}}{\genfrac{[}{]}{0pt}{}{k+l}{l}}
=\sum_{l=1}^n\frac{1+q^l}{[l]_q^{2a+3}}\frac{\genfrac{[}{]}{0pt}{}{n}{l}}{\genfrac{[}{]}{0pt}{}{n+l}{l}}\,
q^{l^2+(a+1)l},
\end{split}
\end{equation*}
as required.

Finally, to prove (\ref{t2(2)}), we  rewrite it in the form
\begin{equation*}
H^{\star}_n[\{2\}^a,1,\{2\}^b]=-\sum_{k=1}^n\frac{(-1)^kA_{n,k}q^{(a+b+1)k}}{[k]_q^{2(a+b)+1}}-\sum_{k=1}^n
\frac{A_{n,k}q^{k(k+1)/2+ak}}{[k]_q^{2a+1}}\,V_{k-1}(2b),
\end{equation*}
where
\begin{equation*}
A_{n,k}=(1+q^k)\frac{\genfrac{[}{]}{0pt}{}{n}{k}}{\genfrac{[}{]}{0pt}{}{n+k}{k}}\,q^{k(k-1)/2}, \quad \text{and}\quad
V_k(2s)=\sum_{j=1}^k\frac{(-1)^j(1+q^j)q^{sj-j(j+1)/2}}{[j]_q^{2s}},
\end{equation*}
and
proceed by induction on $n.$ For $n=1$ the formula is true, since $H^{\star}_1[\{2\}^a,1,\{2\}^b]=q^{a+b+1}.$ For $n>1$ we
use the equality
$$
H^{\star}_n[\{2\}^a,1,\{2\}^b]=\sum_{l=0}^a\frac{q^{n(a-l)}}{[n]_q^{2(a-l)}}\,H^{\star}_{n-1}[\{2\}^l,1,\{2\}^b]
+\frac{q^{(a+1)n}}{[n]_q^{2a+1}}\,H^{\star}_n[\{2\}^b]
$$
and apply the induction assumption  and formula (\ref{t1(2)})  to get
\begin{equation*}
\begin{split}
&H^{\star}_n[\{2\}^a,1,\{2\}^b]=-\sum_{l=0}^a\frac{q^{n(a-l)}}{[n]_q^{2(a-l)}}\sum_{k=1}^{n-1}\frac{(-1)^kA_{n-1,k}\,q^{(b+l+1)k}}{[k]_q^{2(b+l)+1}} \\[5pt]
&\quad -\sum_{l=0}^a\frac{q^{n(a-l)}}{[n]_q^{2(a-l)}}\sum_{k=1}^{n-1}\frac{A_{n-1,k}\,q^{k(k+1)/2+lk}\,V_{k-1}[2b]}{[k]_q^{2l+1}}
-\frac{q^{(a+1)n}}{[n]_q^{2a+1}}\sum_{k=1}^n\frac{(-1)^kA_{n,k}\,q^{bk}}{[k]_q^{2b}}.
\end{split}
\end{equation*}
Changing the order of summation, summing over $l,$ and noting that
\begin{equation} \label{t1(3)}
A_{n-1,k}\sum_{l=0}^a\frac{[n]_q^{2l}}{[k]_q^{2l}}\,q^{(k-n)l}=
A_{n,k}\left(\frac{[n]_q^{2a}}{[k]_q^{2a}}\,q^{(k-n)a}-\frac{[k]_q^2}{[n]_q^2}\,q^{n-k}\right)
\end{equation}
 we obtain
\begin{equation} \label{t2(3)}
\begin{split}
H^{\star}_n[\{2\}^a,1,\{2\}^b]&=-\sum_{k=1}^n\frac{(-1)^kA_{n,k}\,q^{(a+b+1)k}}{[k]_q^{2(a+b)+1}}-\sum_{k=1}^n
\frac{A_{n,k}\,q^{k(k+1)/2+ak}\,V_{k-1}[2b]}{[k]_q^{2a+1}} \\[3pt]
&\quad
+\frac{q^{(a+1)n}}{[n]_q^{2a+2}}\sum_{k=1}^nA_{n,k}\,q^{k(k-1)/2}\,[k]_qV_{k-1}[2b] \\[3pt]
&\quad +\frac{q^{(a+1)n}}{[n]_q^{2a+2}}\sum_{k=1}^n\frac{(-1)^kA_{n,k}\,q^{bk}}{[k]_q^{2b-1}}
-\frac{q^{(a+1)n}}{[n]_q^{2a+1}}\sum_{k=1}^n\frac{(-1)^kA_{n,k}\,q^{bk}}{[k]_q^{2b}}.
\end{split}
\end{equation}
Noting that by (\ref{i2}),
$$
(1+q^j)\sum_{k=j+1}^nA_{n,k}q^{k(k-1)/2}\,[k]_q=A_{n,j}([n]_q-[j]_q)q^{j(j+1)/2},
$$
we can simplify  the double  sum on the right-hand side of (\ref{t2(3)}) to get
\begin{equation} \label{t2(4)}
\begin{split}
&\sum_{k=1}^nA_{n,k}q^{k(k-1)/2}\,[k]_qV_{k-1}[2b]=\sum_{j=1}^n\frac{(-1)^j(1+q^j)q^{bj-j(j+1)/2}}{[j]_q^{2b}}
\sum_{k=j+1}^nA_{n,k}q^{k(k-1)/2}\,[k]_q \\[3pt]
&\qquad =\sum_{j=1}^n\frac{(-1)^jA_{n,j}\,q^{bj}([n]_q-[j]_q)}{[j]_q^{2b}}
=[n]_q\sum_{j=1}^n\frac{(-1)^jA_{n,j}\,q^{bj}}{[j]_q^{2b}}-\sum_{j=1}^n\frac{(-1)^jA_{n,j}\,q^{bj}}{[j]_q^{2b-1}}.
\end{split}
\end{equation}
Now from (\ref{t2(3)}) and (\ref{t2(4)}) we conclude the proof.
\end{proof}

The next two theorems generalize identities (\ref{t1(2)})--(\ref{t2(2)}) to strings of arbitrary collections of twos and ones.

Let $m, n$ be nonnegative integers and ${\bf s}=(s_1, \ldots, s_m)\in {\mathbb Z}^m,$ ${\bf t}=(t_1, \ldots, t_m)\in {\mathbb Z}^m.$
Define the multiple nested sum
\begin{equation*}
\begin{split}
\widehat{\mathcal{H}}_n[{\bf s}; {\bf t}]&=\widehat{\mathcal{H}}_n[s_1,\ldots,s_m;t_1,\ldots,t_m] \\[3pt]
&=\sum_{k=1}^n\frac{\genfrac{[}{]}{0pt}{}{n}{k}}{\genfrac{[}{]}{0pt}{}{n+k}{k}}
\frac{q^{k^2+(t_1-1)k}(1+q^k)}{[k]_q^{s_1}}\,\mathcal{H}_{k-1}[s_2,\ldots,s_m;
t_2,\ldots,t_m],
\end{split}
\end{equation*}
and put $\widehat{\mathcal{H}}_n[{\bf s}; {\bf t}]=0$ if $n<m,$ and
$\widehat{\mathcal{H}}_n[\emptyset]=1$ for all $n\ge 0$ and $m=0.$

\begin{theorem} \label{t2}
Let $m\in {\mathbb N},$ $s_1, \ldots, s_m\in {\mathbb N}_0.$ Then for any positive integer $n,$
\begin{equation} \label{t2.1}
H^{\star}_n[\{2\}^{s_1}, 1, \{2\}^{s_2}, 1, \ldots, \{2\}^{s_m},1]=\sum_{{\bf p}=(2s_1+1)\circ(2s_2+1)\circ\ldots\circ(2s_m+1)}
\widehat{\mathcal{H}}_n[{\bf p}; {\bf \widetilde{p}}],
\end{equation}
where $\circ$ is either comma or plus, and the string ${\bf \widetilde{p}}:=(s_1+1)\circ(s_2+1)\circ\ldots\circ(s_m+1)$
is associated with the string ${\bf p}.$ This means that the choice of commas and pluses in ${\bf p}=(2s_1+1)\circ(2s_2+1)\circ\ldots\circ(2s_m+1)$ and ${\bf \widetilde{p}}=(s_1+1)\circ(s_2+1)\circ\ldots\circ(s_m+1)$ is the same.
\end{theorem}

\begin{proof} Note that for $m=1,$ the theorem is true by (\ref{t2(1)}). We proceed by induction on $n+m.$ If $n=1,$ then
$H^{\star}_1[\{2\}^{s_1}, 1,  \ldots, \{2\}^{s_m},1]=q^{s_1+\cdots+s_m+m}.$ On the other hand, $\widehat{\mathcal{H}}_1[p_1,\ldots,p_r; \widetilde{p_1},\ldots,\widetilde{p_r}]=0$ if $r>1$ and $\widehat{\mathcal{H}}_1[p_1; \widetilde{p_1}]=\widehat{\mathcal{H}}_1[2s_1+\cdots+2s_m+m;
s_1+\cdots+s_m+m]=q^{s_1+\cdots+s_m+m},$ and therefore the equality in (\ref{t2.1}) holds trivially for $n=1$ and any $m\ge 1.$

Now assume that $n>1,$ $m\ge 2$ and apply the expansion
\begin{equation*}
\begin{split}
H^{\star}_n[\{2\}^{s_1}, 1,  \ldots, \{2\}^{s_m},1]&=\sum_{l=0}^{s_1}\frac{q^{n(s_1-l)}}{[n]_q^{2(s_1-l)}}\,H^{\star}_{n-1}[\{2\}^{l}, 1, \{2\}^{s_2}, 1, \ldots, \{2\}^{s_m},1] \\[3pt]
&\qquad +\frac{q^{(s_1+1)n}}{[n]_q^{2s_1+1}}\,H^{\star}_n[\{2\}^{s_2}, 1, \ldots, \{2\}^{s_m},1],
\end{split}
\end{equation*}
then by induction, we have
\begin{equation} \label{Eq01}
\begin{split}
H^{\star}_n[\{2\}^{s_1}, 1,  \ldots, \{2\}^{s_m},1]&=\sum_{l=0}^{s_1}\frac{q^{n(s_1-l)}}{[n]_q^{2(s_1-l)}}\sum_{{\bf p}=(2l+1)\circ(2s_2+1)\circ
\ldots\circ(2s_m+1)}\widehat{\mathcal{H}}_{n-1}[{\bf p}; {\bf \widetilde{p}}] \\
&\qquad +\frac{q^{(s_1+1)n}}{[n]_q^{2s_1+1}}\sum_{{\bf p}=(2s_2+1)\circ\ldots\circ(2s_m+1)}\widehat{\mathcal{H}}_n[{\bf p}; {\bf \widetilde{p}}].
\end{split}
\end{equation}
Next, note that the set of all strings $(2l+1)\circ(2s_2+1)\circ\ldots\circ(2s_m+1)$ falls naturally into two nonintersecting classes $K_1$ and $K_2$
corresponding to the fixed choice of the first $\circ$ as comma or as plus, respectively. Then any string from $K_1$ has the form ${\bf p}=
(2l+1,p_2,\ldots,p_r),$ where $(p_2,\ldots,p_r)=(2s_2+1)\circ\ldots\circ(2s_m+1),$ and any string from the second class $K_2$ is given by
${\bf p}=(2l+s_2+2+t_1,t_2,\ldots,t_u),$ where $(t_1,t_2,\ldots,t_u)=0\circ(2s_3+1)\circ\ldots\circ(2s_m+1).$

Now considering the double sum from (\ref{Eq01}) and splitting the inner sum into two parts in accordance with the above subdivision of strings, we have
\begin{equation} \label{Eq02}
\begin{split}
&\sum_{{\bf p}=(2l+1)\circ(2s_2+1)\circ\ldots\circ(2s_m+1)}
\widehat{\mathcal{H}}_{n-1}[{\bf p};\widetilde{\bf p}]
=\sum_{{\bf p}=(2s_2+1)\circ\ldots\circ(2s_m+1)}
\widehat{\mathcal{H}}_{n-1}[2l+1,{\bf p};l+1,\widetilde{\bf p}] \\
&\quad +\!\sum_{{\bf t}=0\circ(2s_3+1)\circ\ldots\circ(2s_m+1)}
\widehat{\mathcal{H}}_{n-1}[2l+2s_2+2+t_1,t_2,\ldots,t_u;l+s_2+2+\widetilde{t}_1,\widetilde{t}_2,\ldots,\widetilde{t}_u],
\end{split}
\end{equation}
where $\widetilde{0}=0.$ Substituting (\ref{Eq02}) in (\ref{Eq01}) and summing  over $l$ by the formula
\begin{equation*}
\frac{\genfrac{[}{]}{0pt}{}{n-1}{k}}{\genfrac{[}{]}{0pt}{}{n-1+k}{k}}\sum_{l=0}^{s_1}
\frac{[n]_q^{2l}}{[k]_q^{2l}}\,q^{(k-n)l}=\frac{\genfrac{[}{]}{0pt}{}{n}{k}}{\genfrac{[}{]}{0pt}{}{n+k}{k}}\left(
\frac{[n]_q^{2s_1}}{[k]_q^{2s_1}}\,q^{(k-n)s_1}-\frac{[k]_q^2}{[n]_q^2}\,q^{n-k}\right),
\end{equation*}
we obtain
\begin{equation} \label{Eq03}
\begin{split}
&H^{\star}_n[\{2\}^{s_1}, 1,  \ldots, \{2\}^{s_m},1]=\sum_{{\bf p}=(2s_2+1)\circ\ldots\circ(2s_m+1)}\widehat{\mathcal{H}}_n[2s_1+1,{\bf p}; s_1+1,
\widetilde{\bf p}] \\
&\,\,\, -\frac{q^{n(s_1+1)}}{[n]_q^{2s_1+2}}\!\sum_{{\bf p}=(2s_2+1)\circ\ldots\circ(2s_m+1)}\widehat{\mathcal{H}}_n[-1,{\bf p}; 0, \widetilde{\bf p}]
+\frac{q^{(s_1+1)n}}{[n]_q^{2s_1+1}}\!\sum_{{\bf p}=(2s_2+1)\circ\ldots\circ(2s_m+1)}\widehat{\mathcal{H}}_n[{\bf p}; \widetilde{\bf p}]
\\[3pt]
&\,\,\, +\!\sum_{{\bf t}=0\circ(2s_3+1)\circ\ldots\circ(2s_m+1)}\widehat{\mathcal{H}}_n[2s_1+2s_2+2+t_1,t_2,\ldots,t_u;s_1+s_2+2+\widetilde{t}_1,
\widetilde{t}_2,\ldots,\widetilde{t}_u]\\[2pt]
&\,\,\, -\frac{q^{n(s_1+1)}}{[n]_q^{2s_1+2}}
\sum_{{\bf t}=0\circ(2s_3+1)\circ\ldots\circ(2s_m+1)}\widehat{\mathcal{H}}_n[2s_2+t_1,t_2,\ldots,t_u;s_2+1+\widetilde{t}_1,
\widetilde{t}_2,\ldots,\widetilde{t}_u].
\end{split}
\end{equation}
Noticing that the first and fourth sums on the right-hand side of (\ref{Eq03}) add up to
\begin{equation*}
\sum_{{\bf p}=(2s_1+1)\circ(2s_2+1)\circ\ldots\circ(2s_m+1)}
\widehat{\mathcal{H}}_n[{\bf p};\widetilde{\bf p}],
\end{equation*}
 we have
\begin{equation} \label{Eq04}
\begin{split}
 &\,H^{\star}_n[\{2\}^{s_1}, 1,  \ldots, \{2\}^{s_m},1]-\!\!\sum_{{\bf p}=(2s_1+1)\circ\ldots\circ(2s_m+1)}
\widehat{\mathcal{H}}_n[{\bf p};\widetilde{\bf p}]
 \\[2pt]
&= \frac{q^{(s_1+1)n}}{[n]_q^{2s_1+1}}\!\sum_{{\bf p}=(2s_2+1)\circ\ldots\circ(2s_m+1)}\widehat{\mathcal{H}}_n[{\bf p}; \widetilde{\bf p}]
-\frac{q^{n(s_1+1)}}{[n]_q^{2s_1+2}}\!\sum_{{\bf p}=(2s_2+1)\circ\ldots\circ(2s_m+1)}\widehat{\mathcal{H}}_n[-1,{\bf p}; 0, \widetilde{\bf p}]
\\[3pt]
&\quad\,\,\, -\frac{q^{n(s_1+1)}}{[n]_q^{2s_1+2}}
\sum_{{\bf t}=0\circ(2s_3+1)\circ\ldots\circ(2s_m+1)}\widehat{\mathcal{H}}_n[2s_2+t_1,t_2,\ldots,t_u;s_2+1+\widetilde{t}_1,
\widetilde{t}_2,\ldots,\widetilde{t}_u].
\end{split}
\end{equation}
Finally, expanding $\widehat{\mathcal{H}}_n[-1,{\bf p}; 0, \widetilde{\bf p}],$ rearranging the order of summation, and applying (\ref{i2}),
we obtain
\begin{equation*}
\begin{split}
\widehat{\mathcal{H}}_n&[-1,{\bf p}; 0, \widetilde{\bf p}]=\sum_{k=1}^n\frac{\genfrac{[}{]}{0pt}{}{n}{k}}{\genfrac{[}{]}{0pt}{}{n+k}{k}}
\frac{q^{k^2-k}}{[k]_q^{-1}}\,(1+q^k)\sum_{l=1}^{k-1}\frac{q^{(\widetilde{p}_1-1)l}(1+q^l)}{[l]_q^{p_1}}\,\mathcal{H}_{l-1}[p_2,\ldots,p_r;\widetilde{p}_2,
\ldots,\widetilde{p}_r]\\
&=\sum_{l=1}^{n}\frac{q^{(\widetilde{p}_1-1)l}(1+q^l)}{[l]_q^{p_1}}\,\mathcal{H}_{l-1}[p_2,\ldots,p_r;\widetilde{p}_2,
\ldots,\widetilde{p}_r]\sum_{k=l+1}^n\frac{\genfrac{[}{]}{0pt}{}{n}{k}}{\genfrac{[}{]}{0pt}{}{n+k}{k}}
[k]_qq^{k^2-k}(1+q^k)\\
&=[n]_q\widehat{\mathcal{H}}_n[{\bf p};\widetilde{\bf p}]-\widehat{\mathcal{H}}_n[p_1-1,p_2,\ldots,p_r; \widetilde{p}_1,\widetilde{p}_2,\ldots,\widetilde{p}_r].
\end{split}
\end{equation*}
Substituting the last expression in (\ref{Eq04}) and simplifying, we have
\begin{equation} \label{Eq05}
\begin{split}
H^{\star}_n&[\{2\}^{s_1}, 1,  \ldots, \{2\}^{s_m},1]-\!\!\sum_{{\bf p}=(2s_1+1)\circ\ldots\circ(2s_m+1)}
\widehat{\mathcal{H}}_n[{\bf p};\widetilde{\bf p}]
 \\[2pt]
&=\frac{q^{n(s_1+1)}}{[n]_q^{2s_1+2}}\sum_{{\bf p}=(2s_2+1)\circ\ldots\circ(2s_m+1)}
\widehat{\mathcal{H}}_n[p_1-1,p_2,\ldots,p_r; \widetilde{p}_1,\widetilde{p}_2,\ldots,\widetilde{p}_r] \\
&\,\,\,-\frac{q^{n(s_1+1)}}{[n]_q^{2s_1+2}}
\sum_{{\bf t}=0\circ(2s_3+1)\circ\ldots\circ(2s_m+1)}\widehat{\mathcal{H}}_n[2s_2+t_1,t_2,\ldots,t_u;s_2+1+\widetilde{t}_1,
\widetilde{t}_2,\ldots,\widetilde{t}_u].
\end{split}
\end{equation}
It is easy to see that the last two sums in (\ref{Eq05}) are equal and we obtain the required identity.
\end{proof}

In the next theorem, we consider two-one strings ending with $2.$

Let $m$ be a nonnegative integer and ${\bf s}=(s_1,\ldots,s_m)\in{\mathbb Z}^m,$ ${\bf t}=(t_1,\ldots,t_m)\in{\mathbb Z}^m.$ We define the nested sum
\begin{equation*}
\overline{\mathcal{H}}_n[{\bf s}, {\bf t}]=
\sum_{n\ge k_1>\ldots>k_m\ge 1}(-1)^{k_m}\frac{\genfrac{[}{]}{0pt}{}{n}{k_1}}{\genfrac{[}{]}{0pt}{}{n+k_1}{k_1}}\,
q^{k_1^2-k_m(k_m-1)/2}
\prod_{j=1}^m\frac{(1+q^{k_j}) q^{(t_j-1)k_j}}{[k_j]_q^{s_j}}
\end{equation*}
and put  $\overline{\mathcal{H}}_n[{\bf s}, {\bf t}]=0$ if $n<m,$ and $\overline{\mathcal{H}}_n[\emptyset]=1$ for all $n\ge 0$ and $m=0.$

\begin{theorem} \label{t3}
Let $m, s_1, \ldots, s_m\in {\mathbb N}_0,$ $s_{m+1}\in {\mathbb N}.$  Then for any positive integer $n,$
\begin{equation} \label{t3.1}
H^{\star}_n[\{2\}^{s_1}, 1,  \ldots, \{2\}^{s_m},1, \{2\}^{s_{m+1}}]=-\sum_{{\bf p}=(2s_1+1)\circ\ldots\circ(2s_m+1)\circ(2s_{m+1})}
\overline{\mathcal{H}}_n[{\bf p}; {\bf \widetilde{p}}],
\end{equation}
where $\circ$ is either comma or plus, and the string ${\bf \widetilde{p}}:=(s_1+1)\circ\ldots\circ(s_m+1)\circ(s_{m+1})$
is associated with the string ${\bf p}.$ This means that the choice of commas and pluses in ${\bf p}=(2s_1+1)\circ\ldots\circ(2s_m+1)\circ(2s_{m+1})$ and ${\bf \widetilde{p}}=(s_1+1)\circ\ldots\circ(s_m+1)\circ(s_{m+1})$ is the same.
\end{theorem}

Note that for $m=0,$ identity (\ref{t3.1}) coincides with (\ref{t1(2)}). For $m=1,$ the theorem becomes~(\ref{t2(2)}).
For $m\ge 2,$ the proof of Theorem \ref{t3} follows exactly by the same reasoning  as the proof of Theorem \ref{t2} and is left to the reader.

{\it Remark.}
Letting $q\to 1$ in (\ref{t2.1}), we get Zhao's identity  \cite[Theorem 2.2]{Zhao:13}:
\begin{equation} \label{Eq06}
H^{\star}_n(\{2\}^{s_1}, 1, \ldots, \{2\}^{s_m},1)=\sum_{{\bf p}=(2s_1+1)\circ\ldots\circ(2s_m+1)}2^{l({\bf p})}
\widehat{H}_n({\bf p}),
\end{equation}
where
\begin{equation*}
\begin{split}
H^{\star}_n({\bf p})&=H^{\star}_n(p_1,p_2,\ldots,p_r)=\sum_{n\ge k_1\ge k_2\ge \ldots\ge k_r\ge 1}\frac{1}{k_1^{p_1}k_2^{p_2}\cdots k_r^{p_r}}, \\
\widehat{H}_n({\bf p})&=\widehat{H}_n(p_1,p_2,\ldots,p_r)=\sum_{n\ge k_1>k_2>\ldots>k_r\ge 1}
\frac{\binom{n}{k_1}}{\binom{n+k_1}{n}}\frac{1}{k_1^{p_1}k_2^{p_2}\cdots k_r^{p_r}}
\end{split}
\end{equation*}
are ordinary multiple harmonic sums, and $l({\bf p})$  is the length of the string ${\bf p}.$

Similarly, letting $q\to 1$ in (\ref{t3.1}), we get a multiple harmonic sum identity from \cite[Theorem 2.4]{Zhao:13} for two-one strings ending with $2:$
\begin{equation} \label{Eq07}
H^{\star}_n(\{2\}^{s_1}, 1, \ldots, \{2\}^{s_m},1, \{2\}^{s_{m+1}})=\sum_{{\bf p}=(2s_1+1)\circ\ldots\circ(2s_m+1)\circ(2s_{m+1})}2^{l({\bf p})}\,
\overline{\!H}_n({\bf p}),
\end{equation}
where
\begin{equation*}
\overline{\!H}_n({\bf p})=\overline{\!H}_n(p_1,p_2,\ldots,p_r)=\sum_{n\ge k_1>k_2>\ldots>k_r\ge 1}
\frac{\binom{n}{k_1}}{\binom{n+k_1}{n}}\frac{(-1)^{k_r-1}}{k_1^{p_1}k_2^{p_2}\cdots k_r^{p_r}},
\end{equation*}
and $l({\bf p})=\sigma({\bf p})+1$  is the length of the string ${\bf p},$ and
$\sigma({\bf p})$ is the number of commas in ${\bf p}=(2s_1+1)\circ\ldots\circ(2s_m+1)\circ(2s_{m+1}).$

\section{Applications to $q$-zeta values}

In this section, we prove Theorem \ref{T1} and Theorem \ref{T2}. For this purpose, we need a lemma that justifies limit transition from finite
$q$-binomial identities for multiple harmonic sums to corresponding relations for multiple $q$-zeta values.

\begin{lemma} \label{last}
Let $0<q<1,$ $c, c_1, c_2\in {\mathbb R},$ $c>0,$ and let $R_k$ be a sequence of real numbers satisfying $|R_k|<k^{c_1}q^{c_2k}$ for all $k=1,2,\ldots$.
Then
\begin{equation*}
\lim_{n\to\infty}\sum_{k=1}^nq^{ck^2}\left(1-\frac{\genfrac{[}{]}{0pt}{}{n}{k}}{\genfrac{[}{]}{0pt}{}{n+k}{k}}\right)R_k=0.
\end{equation*}
\end{lemma}
\begin{proof}
First, we notice that for $1\le k< n/2$ and $n$ sufficiently large,
\begin{equation*}
1-\frac{\genfrac{[}{]}{0pt}{}{n}{k}}{\genfrac{[}{]}{0pt}{}{n+k}{k}}=1-\frac{(q;q)_n^2}{(q;q)_{n-k}(q;q)_{n+k}}=
1-\frac{(1-q^{n-k+1})(1-q^{n-k+2})\cdots(1-q^n)}{(1-q^{n+1})(1-q^{n+2})\cdots(1-q^{n+k})}=O(q^{c_3n}),
\end{equation*}
where $c_3$ is some positive constant independent of $n.$ Therefore, we have
\begin{equation} \label{Eq08}
\left|\sum_{k=1}^{n/2}q^{ck^2}\left(1-\frac{\genfrac{[}{]}{0pt}{}{n}{k}}{\genfrac{[}{]}{0pt}{}{n+k}{k}}\right)R_k\right|
<O(q^{c_3n})\sum_{k=1}^{n/2}q^{ck^2}|R_k|<O(q^{c_3n})\sum_{k=1}^{n/2}k^{c_2}q^{ck^2+c_1k}.
\end{equation}
On the other side, for $n/2\le k\le n,$ we can apply the trivial inequality
\begin{equation*}
0<1-\frac{\genfrac{[}{]}{0pt}{}{n}{k}}{\genfrac{[}{]}{0pt}{}{n+k}{k}}=
1-\frac{(1-q^{n-k+1})(1-q^{n-k+2})\cdots(1-q^n)}{(1-q^{n+1})(1-q^{n+2})\cdots(1-q^{n+k})}<1
\end{equation*}
to obtain
\begin{equation} \label{Eq09}
\left|\sum_{k=n/2}^{n}q^{ck^2}\left(1-\frac{\genfrac{[}{]}{0pt}{}{n}{k}}{\genfrac{[}{]}{0pt}{}{n+k}{k}}\right)R_k\right|<
q^{cn^2/4}\sum_{k=n/2}^n|R_k|<n^{c_4}q^{cn^2/4+c_5n},
\end{equation}
where $c_4, c_5$ are some real constants independent of $n.$ Now letting $n$ tend to infinity, by (\ref{Eq08}) and (\ref{Eq09}), we conclude the proof.
\end{proof}

Now Theorem \ref{T1} and Theorem \ref{T2} easily follow from Theorem \ref{t2} and Theorem \ref{t3} by Lemma \ref{last}.

\vspace{0.3cm}

{\bf \small Acknowledgement.} We would like to thank Wadim Zudilin for drawing our attention to Zhao's paper \cite{Zhao:13}.

\end{document}